\documentclass{amsart}
\usepackage{amssymb, hyperref}
\usepackage{color}

\newtheorem{thm}{Theorem}
\newtheorem{prp}{Proposition}

\newtheorem{lma}{Lemma}

\theoremstyle{definition}
\newtheorem{df}{Definition}

\newtheorem{ex}{Example}

\def\cvx{{\rm cvx}}

\begin{document}
\title{Convex and subharmonic functions on graphs}

\author{M.J. Burke}
\address{Spring Hill College, 4000 Dauphin Street, Mobile, Alabama, 36608-1791}
\email{mjburke@shc.edu}

\author{T.L. Perkins}
\address{Spring Hill College, 4000 Dauphin Street, Mobile, Alabama, 36608-1791}
\email{tperkins@shc.edu}
\date{\today}

\begin{abstract}
We explore the relationship between convex and subharmonic functions on discrete sets. Our principal
concern is to determine the setting in which a convex function is necessarily subharmonic. We initially
consider the primary notions of convexity on graphs and show that more structure is needed to establish
the desired result. To that end, we consider a notion of convexity defined on lattice-like graphs
generated by normed abelian groups. For this class of graphs, we are able to prove that all convex
functions are subharmonic.
\end{abstract}

\maketitle
\thispagestyle{empty}

\section{Introduction}

Classical analysis provides several equivalent definitions of a convex function, which have led to several non-equivalent concepts of a convex function on a graph.  As an interesting alternative, there appears to be a consensus on how to define subharmonic functions on graphs.  In the real variable counterpart, all convex functions are subharmonic.  It is the aim of this paper to investigate this relationship in the discrete setting.

We show that in the setting of weighted graphs over a normed abelian group one can prove analogs of some classical analysis theorems relating convexity to subharmonic functions.  In particular, (Theorem \ref{T:cvx=>sub}) all convex functions are subharmonic, (Lemma \ref{L:dist-to-a-point=cvx}) for a fixed point $a\in X$, the distance function $d(x,a)$ is convex, and (Propositions \ref{P:dist-cvx=>cvx} and \ref{P:cvx=>dist-cvx}) that a set $F$ is convex if and only if the distance function $d(x,F) = \inf_{y\in F}d(x,y)$ is subharmonic.

For a discrete set with metric, there is generally one straight forward way to define convex sets and convex functions on them. For completeness and ease of reference we present these in Section \ref{S:fund-conp}.  The definitions we give (or something equivalent to them) can be traced back at least to $d$-convexity \cite{GSS73, S72} and $d$-convex functions \cite{SS79}, and possibly much earlier.   Graphs admit a natural metric, i.e. length of the shorted path between two vertices, which leads to one notion of convexity on graphs studied in \cite{S83,S91}.  The notion of $d$-convexity on graphs when $d$ is the standard graph metric is equivalent to the more common notion of geodesic convexity \cite{CMOP05, FJ86}.

Common to \cite{CMOP05, FJ86, S83, S91}, one starts with a graph and then puts a convexity theory on it, by using the graph metric.  However in Section \ref{S:graph-metric} we show that convex sets and functions defined on graphs with respect to the graph metric have a few pleasant but mostly a large number of undesirable properties.  Thereby breaking the analogy with their classical analysis counterparts.

Another approach taken here in Section \ref{S:background_metric} is to allow the vertices themselves to have some underlying structure, e.g. a normed abelian group, and force the edges to be compatible with this metric.  (As opposed to making a metric compatible with the edges.)  In the setting of a normed abelian group there are many notions of a convex functions, see \cite{K04} and references therein.  One introduced in \cite{K04} provides a natural extension of geodesic convexity that makes use of the additional abelian group structure.  In this setting convex and subharmonic functions are of particular interest to image analysis, e.g. \cite{K04, K05}.  In this setting we are able to prove theorems analogous to several standard results from classical analysis.

In particular, (Theorem \ref{T:cvx=>sub}) all convex functions are subharmonic, (Lemma \ref{L:dist-to-a-point=cvx}) for a fixed point $a\in X$, the distance function $d(x,a)$ is convex, and (Propositions\ref{P:dist-cvx=>cvx} and \ref{P:cvx=>dist-cvx}) that a set $F$ is convex if and only if the distance function $d(x,F) = \inf_{y\in F}d(x,y)$ is subharmonic.

\section{Fundamental concepts}\label{S:fund-conp}
We will always assume that a graph is locally finite.

\subsection{Convexity}

Let $X$ be an at most countable set with a metric $d$, i.e. $d\colon X\times X \rightarrow \mathbb{R}$ with the properties
\begin{enumerate}
\item $d(x,y)\ge 0$ all $x,y \in X$ with $d(x,y)=0$ if and only if $x=y$,
\item $d(x,y)=d(y,x)$, and
\item $d(x,y)\le d(x,z)+d(z,y)$.
\end{enumerate}

Traditionally a set $A$ is convex if for all points $x,y\in A$ every point on the line segment connecting them is also in $A$.  Notice that a point $z$ is on the line segment connecting $x,y\in A$ if and only if $d(x,y)=d(x,z)+d(z,y)$.  Hence we take the following definitions:

For $A\subset X$ define
\[c_1(A) = \{z\in X\colon d(x,y)=d(x,z)+d(z,y) \text{ for some } x,y\in A\} \]
when $A=\emptyset$, take $c_1(\emptyset)=\emptyset$, and inductively
$c_n(A)=c_1(c_{n-1}(A))$.
Note that $0=d(x,x)=d(x,x)+d(x,x)$, hence $A\subset c_1(A)\subset \cdots \subset c_n(A)$ for all $n$.

\begin{df}
Let $A\subset X$.  The \emph{convex hull} of $A$ is
\[\cvx(A) = \bigcup_{n=1}^\infty c_n(A).\]
Naturally, the set $A$ is said to be \emph{convex} if $\cvx(A)=A$.  Clearly $\emptyset$ and $X$ are convex.

We say that the point $z$ is \emph{in between} $x$ and $y$ whenever $d(x,y)=d(x,z)+d(z,y)$ is satisfied.
\end{df}

Consequently,
\begin{lma}\label{L:a_cvx<=>a=c_1(a)}
A set $A\subset X$ is convex if and only if $A=c_1(A)$.
\end{lma}
\begin{proof}
If $A=c_1(A)$ then $c_2(A) = c_1(c_1(A))=c_1(A)=A$.  Hence by induction $c_n(A)=A$ and so $A=\cup c_n(A) = \cvx(A)$.  Thus $A$ is convex.

Suppose that $A$ is convex.  Then $A=\cvx(A) = \cup c_n(A) \supset c_1(A) \supset A$.  Thus $A=c_1(A)$.
\end{proof}

\begin{prp}
For all sets $A, B\subset X$,
\begin{align}
A &\subset \cvx(A) \\
A\subset B &\Rightarrow \cvx(A)\subset \cvx(B)\\
\cvx(A) &= \cvx(\cvx(A)).
\end{align}
\end{prp}
\begin{proof}
\begin{enumerate}
\item We've already shown that $A\subset c_1(A)\subset \cdots \subset c_n(A)$ for all $n$ and so $A\subset \cup c_n(A)=\cvx(A)$.

\item  For any sets $X$ and $Y$, if $X\subset Y$ then $c_1(X)\subset c_2(Y)$.  Indeed for any $z\in c_1(X)$ there exists by definition $x_1,x_2\in X$ so that $d(x_1,x_2)=d(x_1,z)+d(z,x_2)$, but as $x_1,x_2\in X\subset Y$ this shows that $z\in c_1(Y)$.  Then as $A\subset B$, we have $c_1(A)\subset c_1(B)$.  Then by induction, $c_n(A)\subset c_n(B)$.  Therefore $\cvx(A)\subset \cvx(B)$.

\item  The claim $\cvx(A) = \cvx(\cvx(A))$ amounts to saying that $\cvx(A)$ is convex.  We will use Lemma 1 to show this.  Consider any $z\in c_1(\cvx(A))$.  This means there exists $x,y\in \cvx(A) = \cup c_n(A)$ so that $d(x,y) = d(x,z)+d(z,y)$.  However as $A\subset c_1(A)\subset c_2(A)\subset \cdots \subset c_n(A) \subset \cdots$ we know $x, y\in c_n(A)$ for some $n$, and so $z\in c_1(c_n(A))=c_{n+1}(A)\subset \cvx(A)$.  Hence $c_1(\cvx(A))=\cvx(A)$.
\end{enumerate}
\end{proof}

The following proposition shows that our definition of convex hull is equivalent to the usual one, i.e. the convex hull of $A$ is the intersection of all convex sets that contain $A$.
\begin{prp}
For any $A\subset X$, the set $\cvx(A)$ is the intersection of all convex sets that contain $A$.
\end{prp}
\begin{proof}
Let $B\subset X$ be a convex set containing $A$.  As noted previously $A\subset B$ implies $\cvx(A)\subset \cvx(B)$.  However $\cvx(B)=B$ by hypothesis.  Hence $\cvx(A) \subset B$ for all convex $B$ containing $A$.  Therefore
\[\cvx(A) \subset \bigcap\{B\colon  A\subset B \text{ and } B \text{ convex} \}.\]

As $\cvx(A)$ is convex and $A\subset \cvx(A)$, it must be included in the intersection above.  Thus
\[\bigcap\{B\colon  A\subset B \text{ and } B \text{ convex} \} \subset \cvx(A).\qedhere\]
\end{proof}

\begin{prp}
If $A$ and $B$ are convex, then $A\cap B$ is convex.
\end{prp}

\begin{proof}
Let $A$ and $B$ be convex. Then by Lemma \ref{L:a_cvx<=>a=c_1(a)} $A=c_1(A)$ and $B=c_1(B)$.  We will show that $c_1(A\cap B)= c_1(A)\cap c_1(B)=A\cap B$.  We've already noted that $A\cap B\subset c_1(A\cap B)$.

Suppose that $z\in c_1(A\cap B)$.  Then there exists $x,y\in A\cap B$, such that $d(x,y)=d(x,z) + d(z,y)$.  Hence $z\in c_1(A)$ and $z\in c_1(B)$, that is, $z\in c_1(A)\cap c_1(B)$.  As $A=c_1(A)$ and $B=c_1(B)$, we now have $z\in c_1(A)\cap c_1(B) = A\cap B$.  Therefore $c_1(A\cap B)\subset A\cap B$.  Thus $A\cap B = c_1(A\cap B)$ and so $A\cap B$ is convex.
\end{proof}

\begin{prp}Let $I$ be an ordered set and take $\{A_\alpha\}_{\alpha\in I}$ to be a collection of convex sets in $X$ where $A_\alpha\subset A_\beta$ whenever $\alpha < \beta$ and $\alpha,\beta\in I$. The set formed by taking the union of $A_\alpha$ for $\alpha \in I$ is convex.
\end{prp}

\begin{proof}
We must show that $\cup A_\alpha$ is convex.
Consider the set $c_1(\cup A_\alpha)$.  For any $z\in c_1(\cup A_\alpha)$, we can find $x,y\in \cup A_\alpha$ so that $d(x,y)=d(x,z)+d(z,y)$. However $x,y \in  \cup A_\alpha$ implies that $x \in A_\alpha$ and $y \in  A_\beta$ for some $\alpha, \beta \in I$.  Without loss of generality we assume that $\alpha < \beta$.  By hypothesis, $A_\alpha\subset A_\beta$.  Hence $x,y\in A_\beta$.  Since $z$ satisfies $d(x,y)=d(x,z)+d(z,y)$ for $x,y\in A_\beta$ with $A_\beta$ convex, we see that $z\in c_1(A_\beta)=A_\beta$.  As $z$ was arbitrarily chosen from $c_1(\cup A_\alpha)$, we have $c_1(\cup A_\alpha)\subset \cup A_\alpha$.

By construction the reverse inclusion $\cup A_\alpha\subset c_1(\cup A_\alpha)$ is immediate.  Hence $c_1(\cup A_\alpha)= \cup A_\alpha$.  Recall, Lemma 1,that a set $A$ is convex if and only if $A=c_1(A)$.  Thus $\cup A_\alpha$ is convex.
\end{proof}

\begin{df}
Let $A$ be a convex set. A function $f\colon A\rightarrow \mathbb{R}$ is \emph{convex at the point} $z\in A$ if
\[f(z) \le \frac{d(y,z)}{d(x,y)}f(x)+\frac{d(x,z)}{d(x,y)}f(y)\]
whenever $z$ is in between $x, y\in A$, i.e. $d(x,y)=d(x,z)+d(z,y)$.  A function is said to be \emph{convex on} $A$ if it is convex at every point in $A$.  Furthermore, a function is simply called \emph{convex} when it is convex on the entire set $X$.
\end{df}

The vertices of a graph admit a natural metric defined as the length of the shortest path between them.  With this, the notions of convex and convex functions extend naturally to all graphs, see \cite{CMOP05, FJ86, S83, S91}. 

\subsection{Subharmonic functions on a graph}
Introductions to various aspects of the theory can be found in \cite{BLS07, K05, S94, W94}.

Consider a graph $G$.  The vertices of this graph will be denoted $X$ (to stay consistent with above), which shall be the domain of our (sub)harmonic functions.  A function $f\colon X \rightarrow \mathbb{R}$ is said to be \emph{harmonic} at $x\in X$ if
\[f(x) = \frac{1}{\deg(x)}\sum_{y\sim x}f(y)\]
and {subharmonic} at $x\in X$ if
\[f(x) \le \frac{1}{\deg(x)}\sum_{y\sim x}f(y)\]
where $\deg(x)$ denotes the degree of $x$ and $y\sim x$ means that $y$ is adjacent to $x$.  A function is (sub)harmonic if it is (sub)harmonic at every point $x\in X$. Observe that constant functions are always harmonic (thereby subharmonic too), and so these classes of functions are never empty.

\begin{lma}
If the graph $X$ is connected, regular of degree two and triangle free, then a subharmonicity is the same as convexity.
\end{lma}
\begin{proof}
Each vertex $z$ has only two neighbors $x,y$.  As the graph is triangle free $d(x,y)=2$.  Hence 
\[\frac{1}{deg(z)}\sum_{\zeta\sim z}f(\zeta) = \frac{1}{2}\left(f(x)+f(y)\right) = \frac{d(y,z)}{d(x,y)}f(x)+\frac{d(x,z)}{d(x,y)}f(y)\]

By definition $f$ is subharmonic at $z$ if $f(z)$ is less than or equal to the left side of the equation above and $f$ is convex at $z$ if $f(z)$ is less that or equal to the right side of the equation above.  Therefore subharmonicity and convexity are equivalent conditions.
\end{proof}

We will also use a standard modification of the definition of subharmonic functions on graphs to allow for positive edge weights. Namely, a function $f\colon X\rightarrow \mathbb{R}$ is subharmonic at $x$ if
\[0\le \sum_{y\sim x} e(x,y)[f(y)-f(x)],\]
which with some arithmetic becomes
\[f(x)\le \frac{1}{M_x}\sum_{y\sim x} e(x,y)f(y),\]
where $e(x,y)=e(y,x)\ge 0$ is the edge weight and $M_x=\sum_{y\sim x} e(x,y)$.  If the edge weights are all taken to be one, then this definition is identical to the first.

\section{The distance is given by the graph metric.}\label{S:graph-metric}

In this section we provide two simple theorems which show that for a large class of graphs, convex functions are indeed subharmonic.

\begin{thm}
Let $z$ be a point in $X$.  Suppose that $\deg(z)>1$ and that $z$ is not part of any triangle.  If $f$ is convex at $z$, then $f$ is subharmonic at $z$.  Consequently, if the graph has no triangles or vertices of degree less than $2$, then every convex function is subharmonic.
\end{thm}

\begin{proof}
Let $B=\{y\in X\colon y\sim z\}$ be all the vertices adjacent to $z$.  By hypothesis $\deg(z)=|B|>1$, and so there are at least two vertices $y_1, y_2\in B$.  As $z$ is adjacent to both $y_1$ and $y_2$ and as $z$ is assumed to not be apart of a triangle, $y_1$ is not adjacent to $y_2$.  Hence $z$ is in between $y_1$ and $y_2$, that is, on a geodesic connecting $y_1$ and $y_2$.  In fact, $2=d(y_1, y_2)=d(y_1,z)+d(z,y_2)$ with $d(y_1,z)=d(z,y_2)=1$.  Hence for all $y_1, y_2\in B$ we have
\begin{equation}\label{E:basic}
2f(z) \le f(y_1)+f(y_2)
\end{equation}
by convexity.

Now we sum Equation (\ref{E:basic}) over all unordered pairs of points $y_1, y_2\in B$.  Naturally there are $\binom{\deg(z)}{2}$ such pairs and each vertex $y\in B$ will appear precisely $\deg(z)-1$ times.  (Recall $B=\{y\colon y\sim z\}$ and so $|B| = \deg(z)$.)  Hence
\[\binom{\deg(z)}{2} 2f(z) \le (\deg(z)-1)\sum_{y\sim z}f(y),\]
which simplifies to
\[f(z) \le \frac{1}{\deg(x)}\sum_{y\sim z}f(y).\]
Thus $f$ is subharmonic at $z$.
\end{proof}

Furthermore,
\begin{thm}
Let $z$ be a point in $X$.  If the neighbors of $z$ can be partitioned into pairs such that the vertices in each pair are non-adjacent then a function is convex at $z$ implies that it is also subharmonic at $z$.
\end{thm}

\begin{proof}
For any vertices $y_1, y_2$ in a pairing of the partition of the neighbors of $z$ are non-adjacent, the vertex $z$ must be between them, and hence
\[2f(z) \le f(y_1)+f(y_2)\]
for any function $f$ subharmonic at $z$.  Consequently if we sum this inequality over all $\deg(z)/2$ pairings, we have
\[2 \frac{\deg(z)}{2} f(z) \le \sum_{y\sim z}f(y).\]
Therefore $f$ is subharmonic at $z$.
\end{proof}

Notice that for the standard square lattice both theorems imply that a convex function is subharmonic. If $z$ was connected to an odd number of non-adjacent points then only the first theorem implies that a function convex at $z$ is subharmonic at $z$.  Similarly when the graph is the standard triangular tiling of the plane, only the second theorem would show that every convex function is subharmonic.

\begin{thm}
Let $F$ be any subset of $X$.  If the distance function 
\[d(\cdot, F):=\inf\left\{ d(\cdot, f) \colon f\in F \right\}\]
 is convex, then $F$ is convex.
\end{thm}

\begin{proof}
Consider any point $z\in X$ that lies between $x,y\in F$.  If the distance function is convex, we have
\[0\le d(z,F) \le \frac{d(y,z)}{d(x,y)}d(x,F)+\frac{d(x,z)}{d(x,y)}d(y,F),\]
but $d(x,F)=d(y,F)=0$ as $x,y\in F$.  Therefore $d(z,F)=0$ and so $z$ must also be a point in $F$.
\end{proof}

\begin{ex}
\label{Ex:cycle}
Consider a cycle on four vertices, i.e. $X=\{a,x,y,z\}$ with $a\sim x, x\sim y, y\sim z, z\sim a$. One would easily believe that $F=\{a\}$ is convex.  Hence $d(x,F)=d(z,F)=1$, and $y$ is in between $x$ and $z$.  However
\[2 = d(y,a) \not\le \frac{1}{2} d(x,a) +\frac{1}{2} d(z,a) = 1.\]
Hence $d(\cdot, a)$ is not convex, and certainly not subharmonic.

Observe also the set $\{x,y,z\}$ is NOT convex.  We believe this reveals part of the problem with this definition of convexity.  Namely that a geodesic line segment need not be convex.  It seems that `few' graphs have convex geodesics.  (However $X=\mathbb{Z}$ with $x\sim y$ when $|x-y|=1$, and the standard triangular tiling of the plane are two such.)
\end{ex}

It would seem that more structure is needed to have a workable theory.

\section{Graphs over a normed abelian group.} \label{S:background_metric}
For the remainder of this paper, we consider weighted graphs where the vertex set $X$ is a normed abelian group, and the graph is compatible with the norm. We will denote the norm $||\cdot||$.  Recall that the graph structure is \emph{compatible with the norm} if there is a constant $r>0$ such that $x\sim y$ if and only if $||x-y||\le r$ and the edge weights are given by the norm $e(x,y)=||x-y||\le r$.

In particular, graphs of this type include all lattice graphs.  By rescaling $X$ by $r$ we can always assume without loss of generality that $r=1$.

Graphs of this type pick up a number of traits from analysis.  Far from the least important is a local similarity property.  When one does analysis in a domain $D\subset \mathbb{R}^n$ (or on a manifold) every point $z\in D$ has a neighborhood which is locally like a ball in $\mathbb{R}^n$.  We see the same property here.

This can also be viewed as a translation invariance property, we could translate any point $x_0$ to the origin by taking $X\mapsto X-x_0$ and nothing would change. More explicitly, we denote $B_r(x_0):=\{y\in X:y\sim x_0\}$ and for every $x_0$ in $X$ there is a simple 1-1 correspondence between $B_r(x_0)$ and $B_r(0)$.  If $y\in B_r(x_0)$, then $z=y-x_0\in B_r(0)$, and if $z\in B_r(0)$, then $x_0+z\in B_r(x_0)$.

Furthermore, if $\zeta\in B_r(0)$, then $-\zeta\in B_r(0)$.  Hence
\begin{equation}\label{E:sim_to_zero}
\{y\in X:y\sim x\}:=B_r(x)=\{x+\zeta\colon \zeta\in B_r(0)\}=\{x-\zeta\colon \zeta\in B_r(0)\}
\end{equation}

We maintain the same notion of a convex function, namely
\[||x-y||f(z)\le ||y-z||f(x)+||x-z||f(y),\]
whenever $||x-y||=||x-z||+||z-y||$.  However in this context we can work with midpoints.

In \cite{K96}, Kiselman defines a function $f$ on an abelian group $X$ to be \emph{midpoint convex} if
\[f(x) \le \frac{1}{2}f(x+z)+\frac{1}{2}f(x-z)\]
for all $x$ and $z$ in $X$.  (Actually he uses the notion of upper addition to for functions defined on the extended real line, i.e. $\mathbb{R}\cup\{\pm\infty\}$, but we will not be needing such subtleties here.)  Trivially a convex function is always midpoint convex.

We will now see that this notion of midpoint convexity allows us to achieve our goals.

\begin{thm}\label{T:cvx=>sub}
Consider a weighted graph where the vertex set $X$ is a normed abelian group and the graph is compatible with the norm.  Every midpoint convex function is subharmonic.
\end{thm}
\begin{proof}
Pick any $x\in X$.  Observe that by Equation \ref{E:sim_to_zero}
\begin{align*}\sum_{y\sim x} e(x,y) f(y) &= \frac{1}{2}\sum_{z\in B_r(0)} e(x,x+z) f(x+z) + \frac{1}{2}\sum_{z\in B_r(0)} e(x,x-z) f(x-z) \\
 &= \sum_{z\in B_r(0)} e(x,x+z) \left(\frac{1}{2}f(x+z)+\frac{1}{2}f(x-z)\right).
\end{align*}
Hence by (midpoint) convexity
\[f(x)M_x=f(x)\sum_{z\in B_r(0)}e(x,x+z) \le \sum_{y\sim x} e(x,y) f(y),\]
which shows that $f$ is subharmonic at $x$.
\end{proof}

A set $A\subset X$ is called \emph{convex} if the function
\[\chi_A(x)=\begin{cases} 0 &\colon x\in A, \\+\infty &\colon x\in X\setminus A,\end{cases}\]
is convex, or, equivalently, if $z\in A$ whenever there exists $x,y\in A$ such that $||x-y||=||x-z||+||z-y||$.  This again easily implies midpoint convexity, i.e. if $z\in A$ whenever there is an $x\in X$ such that both $z+x$ and $z-x$ are in $A$

\begin{prp}\label{P:dist-cvx=>cvx}
Let $F$ be any subset of $X$.  If the distance function $d(x,F)=\inf\{||x-y||\colon y\in F\}$ is convex, then the set $F$ is convex.
\end{prp}
\begin{proof}
Let $x\in X$ so that there is some $z\in X$ with $x\pm z\in F$.  Then by midpoint convexity
\[0\le d(x,F) \le \frac{1}{2} d(x+z,F) + \frac{1}{2} d(x-z,F) = 0.\]
Thus $d(x,F)=0$ and so $x\in F$.
\end{proof}

Notice for that for the simple case $F=\{a\}$ we get the converse of the previous result.
\begin{lma}\label{L:dist-to-a-point=cvx}
For any fixed $a\in X$, the function $f(z) = ||z-a||$ is midpoint convex.
\end{lma}
\begin{proof}
This follows immediately from the triangle inequality on the norm.  Indeed, for any $x,y,z\in X$ with $||x-y||=||x-z||+||z-y||$ we have
%
\begin{align*}
2f(x) & = 2||x-a|| = ||2(x-a)|| \\
& = ||(x-a)-z + (x-a)+z|| \\
& \le ||(x-a)-z || + ||(x-a)+z|| \\
& = f(x-z)+f(x+z). \qedhere
\end{align*}

\end{proof}

Of course, the minimum of a two convex functions is in general not convex, which is perhaps one reason why the following  result is interesting.

However in general the classical proofs heavily rely upon the fact that for any point $x$ and convex set $F$ there is always a  unique nearest neighbor $y\in F$ to $x$.

\begin{df}
We say that a set $F$ has the \emph{nearest neighbor} property if for all $y_1, y_2\in F$ and $z\in X$ there exists a $y\in F$ (possibly $y_1$ or $y_2$) such that
\[2 ||y-z|| \le ||y_1+y_2 -2z||.\]
\end{df}

\begin{prp}\label{P:cvx=>dist-cvx}
If $F$ is a convex subset of $X$ with the nearest neighbor property, then the distance function $d(\cdot,F)$ is midpoint convex (and hence subharmonic).
\end{prp}
\begin{proof}
Pick any $z\in X\setminus F$.  We will show that $d(\cdot,F)$ is midpoint convex at $z$.  By replacing $F$ with $F-z$ we may assume without loss of generality that $z=0$.

Clearly it is possible for there to be an $x\in B_r(0)$ such that $d(x,F)\le d(0,F)$.  However by switching to normed abelian groups we've a strong property to use.  Namely that if $x\in B_r(0)$ then $-x\in B_r(0)$.  We will show that for convex sets with the nearest neighbor property, that
\[2d(0,F)\le d(x,F)+d(-x,F),\]
that is to say that $d(\cdot, F)$ is midpoint convex (and hence subharmonic).

We can find $y_1, y_2\in F$ such that $d(x,F)= ||x - y_1||$ and $d(-x,F)= ||(-x) - y_2||$.  Let $y$ be a point in $F$ such that $2||y||\le ||y_1+y_2||$.  Then
\begin{align*}
2d(0, F) &\le 2||y|| \\
& \le ||y_1+y_2||  \\
& = ||y_1+y_2 + x - x||  \\
& = ||(y_1-x) + (y_2+x)|| \\
& \le ||y_1-x|| +||y_2+x|| \\
& =  d(x,F)+d(-x,F). \qedhere
\end{align*}

\end{proof}

\bibliographystyle{amsplain}
\bibliography{convex-ref}

\providecommand{\bysame}{\leavevmode\hbox to3em{\hrulefill}\thinspace}
\providecommand{\MR}{\relax\ifhmode\unskip\space\fi MR }
\providecommand{\MRhref}[2]{%
  \href{http://www.ams.org/mathscinet-getitem?mr=#1}{#2}
}
\providecommand{\href}[2]{#2}
\begin{thebibliography}{10}

\bibitem{BLS07}
Turker Biyikoglu, Josef Leydold, and Peter~F. Stadler, \emph{Laplacian
  eigenvectors of graphs: {P}erron-{F}robenius and {F}aber-{K}rahn type
  theorems}, Lecture Notes in Mathematics, vol. 1915, Springer, 2007.
  \MR{2340484 (2009a:05119)}

\bibitem{CMOP05}
Jos\'{e} C\'{a}ceres, Alberto M\'{a}rquez, Ortrud~R. Oellermann, and
  Mar\'{i}a~Luz Puertas, \emph{Rebuilding convex sets in graphs}, Discrete
  Math. \textbf{297} (2005), no.~1-3, 26--37. \MR{2159429 (2006e:05053)}

\bibitem{FJ86}
Martin Farber and Robert~E. Jamison, \emph{Convexity in graphs and
  hypergraphs}, SIAM Journal Algebraic Discrete Methods \textbf{7} (1986),
  no.~3, 433--444. \MR{0844046 (87i:05166)}

\bibitem{GSS73}
L.~F. German, V.~P. Soltan, and P.~S. Soltan, \emph{Certain properties of
  $d$-convex sets}, Dokl. Akad. Nauk SSSR \textbf{212} (1973), 1276--1279.
  \MR{0333977 (48 \#12296)}

\bibitem{K96}
Christer~O. Kiselman, \emph{Regularity of distance transformations in image
  analysis}, Computer Vision and Image Understanding \textbf{64} (1996), no.~3,
  390--398.

\bibitem{K04}
\bysame, \emph{Convex functions on discrete sets}, Lecture Notes in Computer
  Science, vol. 3322, pp.~443--457, Springer, 2004. \MR{2166029}

\bibitem{K05}
\bysame, \emph{Subharmonic functions on discrete structures}, Progr. Math.,
  vol. 238, pp.~67--80, Birkh{\"a}user Boston, 2005. \MR{2174310 (2007c:31007)}

\bibitem{S94}
Paolo~M. Soardi, \emph{Potential theory on infinite networks}, Lecture Notes in
  Mathematics, vol. 1590, Springer-Verlag, 1994. \MR{1324344 (96i:31005)}

\bibitem{S72}
P.~S. Soltan, \emph{Helly's theorem for $d$-convex sets}, Dokl. Akad. Nauk SSSR
  \textbf{205} (1972), 537--539. \MR{0308935 (46 \#8047)}

\bibitem{S83}
V.~P. Soltan, \emph{$d$-convexity in graphs}, Dokl. Akad. Nauk SSSR
  \textbf{272} (1983), no.~3, 535--537. \MR{0723776 (85a:05077)}

\bibitem{SS79}
V.~P. Soltan and P.~S. Soltan, \emph{$d$-convex functions}, Dokl. Akad. Nauk
  SSSR \textbf{249} (1979), no.~3, 555--558. \MR{0553209 (80j:52001)}

\bibitem{S91}
Valeriu~P. Soltan, \emph{Metric convexity in graphs}, Studia Univ.
  Babes,-Bolyai Math. \textbf{36} (1991), no.~4, 3--43. \MR{1281973
  (95k:52002)}

\bibitem{W94}
Wolfgang Woess, \emph{Random walks on infinite graphs and groups-a survey on
  selected topics}, Bull. London Math. Soc. \textbf{26} (1994), no.~1, 1--60.
  \MR{1246471 (94i:60081)}

\end{thebibliography}
\end{document}